\documentclass[a4paper, 12pt]{amsart}

\usepackage[utf8]{inputenc}
\usepackage[toc]{appendix}

\usepackage[english]{babel}

\usepackage{graphicx}
\usepackage{subfigure}
\usepackage[justification=centering]{caption}

\usepackage{amsmath,a4wide}
\usepackage{amsfonts}
\usepackage{amssymb}
\usepackage{amsthm}
\usepackage{mathtools}

\usepackage{hyperref}

\numberwithin{equation}{section}
\mathtoolsset{showonlyrefs}

\providecommand{\norm}[1]{\lVert#1\rVert}

\newcommand{\R}{\mathbb{R}}

\newcommand{\Z}{\mathbb{Z}}
\newcommand{\N}{\mathbb{N}}
\newcommand{\Lt}[1][d]{L^2(\R^{#1})}
\newcommand{\G}{\mathcal{G}}

\newcommand{\A}{\mathcal{A}}
\renewcommand{\l}{\lambda}
\renewcommand{\L}{\Lambda}

\theoremstyle{plain}
\newtheorem{theorem}{Theorem}[section]

\theoremstyle{plain}

\theoremstyle{plain}
\newtheorem{lemma}[theorem]{Lemma}

\theoremstyle{plain}
\newtheorem{proposition}[theorem]{Proposition}

\theoremstyle{definition}
\newtheorem{definition}[theorem]{Definition}

\theoremstyle{remark}

\theoremstyle{remark}

\theoremstyle{definition}

\title[Minimal Frame Operator Norms]{Minimal Frame Operator Norms via Minimal Theta Functions}
\author{Markus Faulhuber}
\address{NuHAG, Faculty of Mathematics, University of Vienna, Oskar-Morgenstern-Platz 1, 1090 Vienna, Austria}
\email{markus.faulhuber@univie.ac.at}
\thanks{The author wishes to thank Karlheinz Gröchenig for many fruitful discussions on the topic and Thomas Strohmer for pointing out reference \cite{HirGarBor93}. The author gratefully acknowledges the feedback and comments of the referees which helped to improve this work significantly. The author was supported by the Austrian Science Fund (FWF): [P26273-N25].}

\begin{document}

\begin{abstract}
	We investigate sharp frame bounds of Gabor frames with chirped Gaussians and rectangular lattices or, equivalently, the case of the standard Gaussian and general lattices. We prove that for even redundancy and standard Gaussian window the hexagonal lattice minimizes the upper frame bound using a result by Montgomery on minimal theta functions.
\end{abstract}

\subjclass[2010]{42C15, 33E05}
\keywords{Frame Bounds, Gabor Frames, Theta Functions.}

\maketitle

\section{Introduction and Main Result}
The moving spirit of this work originates in a conjecture formulated by Strohmer \& Beaver in 2003 \cite{StrBea03}. They claim that the condition number of the Gabor frame operator for a Gaussian window and a hexagonal lattice of fixed density $\delta > 1$ is minimal among all lattice of same fixed density $\delta$. In their work Strohmer \& Beaver show that the hexagonal lattice is preferable over the quadratic lattice, which, until then, was a candidate for the optimal condition number due to a conjecture by le Floch, Alard \& Berrou in 1995 \cite{FloAlaBer95}. Just recently, Faulhuber \& Steinerberger proved that in the case of integer redundancy the quadratic lattice is indeed optimal, if only rectangular lattices are considered \cite{FauSte16}. In fact, it was shown that the lower and upper frame bound are optimized independently from each other by the square lattice and hence, their ratio, which gives the condition number, is minimized in this case. When it comes to analytic 
investigations of Gabor frame bounds, the work by Janssen starting from 1995 \cite{Jan95, Jan96} is a cornerstone. For the Gaussian window, it turns out that optimizing the frame bounds in the case of even redundancy is equivalent to finding the maximum and the minimum of the heat kernel of the flat Laplacian on the torus $\R^2/\L$ and then optimizing among all lattices of fixed area. Investigations in this direction have been carried out by Montgomery in 1988 \cite{Mon88} and by Baernstein in 1997 \cite{Bae97}. Using Montgomery's theorem about minimal theta functions \cite[Theorem 1]{Mon88}, we prove our main result as given in the abstract. Besides Montgomery's theorem, the main tools for this work are the Fourier transform and the Poisson summation formula. Since, we will apply both only on Gaussians, we do not have to care about conditions for the formulas to hold and omit the technical details. The Fourier transform of a function $f$ is given by
\begin{equation}
	\mathcal{F} f (\omega) = \int_{\R^d} f(t) e^{-2\pi i \omega \cdot t} \, dt.
\end{equation}
The Poisson summation formula is then given by
\begin{equation}
	\sum_{n \in \Z^d} f(n+x) = \sum_{k \in \Z^d} \mathcal{F}f(k)e^{2 \pi i k \cdot x}.
\end{equation}
We will use both formulas for dimensions $d = 1,2$. Equipped with the mentioned tools we will prove the following theorem.

\begin{theorem}[Main Result]\label{theorem_main}
	Let $S$ be the generating matrix for the lattice $\L = S \Z^2$ of fixed, even density $2n$, $n \in \N$ and let $g_0(t) = 2^{1/4} e^{-\pi t^2}$ be the standard Gaussian. Let
	\begin{equation}
		\L_h = S_h \Z^2 = \frac{1}{\sqrt{2n}} \, \left(
	\begin{array}{c c}
		\frac{\sqrt[4]{3}}{\sqrt{2}} & 0\\
		\frac{\sqrt{2}}{2 \sqrt[4]{3}} & \frac{\sqrt{2}}{\sqrt[4]{3}}
	\end{array}
	\right) \Z^2
	\end{equation}
	be the hexagonal lattice. We denote the upper frame bound of $\G(g_0,\L)$ by
	\begin{equation}
		B = B_{g_0,2n}(\L).
	\end{equation}
	Then
	\begin{equation}
		B_{g_0,2n}(\L_h) \leq B_{g_0,2n}(\L)
	\end{equation}
	with equality only for $\L = \widetilde{S}_h \Z^2$ with
	\begin{equation}
		\widetilde{S}_h = Q S_h \mathcal{B}.
	\end{equation}
	Here, $Q$ is an orthogonal matrix and $\mathcal{B} \in SL(2,\Z)$.
%
\end{theorem}

The theorem tells us that the upper frame bound is minimized by a hexagonal lattice. The matrix $\mathcal{B}$ is an element of the modular group $SL(2,\Z)$ which consists of matrices with integer entries and determinant 1. It is well-known that $\Z^2$ is invariant under the action of this group. In fact, it is just another choice for a basis of our lattice. Furthermore, the action of the matrix $Q$ does not change the geometry of the lattice. Therefore, those matrices will be ignored in our proofs and we will focus on lattices generated by lower triangular matrices. For more details on lattices and their generating matrices we refer to the textbook by Conway \& Sloane \cite{ConSlo99}.

This work is structured as follows. In Section \ref{sec_frame} we briefly introduce the concept of Gabor systems and Gabor frames. We state the frame inequality and describe the constants appearing, which are the objects of interest in this work. In Section \ref{sec_Gaussians} we compute sharp frame bounds for Gabor frames of even redundancy with Gaussian window using Janssen's methods in \cite{Jan96}. We show that, starting from a rectangular lattice, the upper frame bound is always lowered by shearing the lattice, or, equivalently, by chirping the window. We will also see, that the upper frame bound is periodic in the shearing (chirping) parameter. In Section \ref{sec_minimal} we show our main result, Theorem \ref{theorem_main}. Finally, in Section \ref{sec_red2} we prove that for redundancy 2 the value of the condition number in the case of the hexagonal lattice conjectured by Strohmer \& Beaver \cite{StrBea03} is correct. This gives an analytic proof that, concerning the frame condition, the hexagonal lattice is preferable over the square lattice.

\section{Gabor Frames}\label{sec_frame}

A Gabor system for the Hilbert space $\Lt[]$ is the set of time-frequency shifted versions of a window function $g \in \Lt[]$ with respect to some index set $\L \subset \R^2$. We denote the Gabor system by
\begin{equation}
	\G(g,\L) = \lbrace \pi(\l) g \, | \, \l \in \L \rbrace.
\end{equation}
Here, $\pi(\l)$ denotes the time-frequency shift operator
\begin{equation}
	\pi(\l)g(t) = M_\omega T_x g(t) = e^{2\pi i \omega \cdot t} g(t-x), \qquad \lambda = (x,\omega) \in \R^{2}.
\end{equation}
The elements of $\G(g,\L)$ are called atoms. Throughout this work, the index set $\L$ will be a lattice in $\R^2$. We can generate any lattice by some (non-unique) invertible $2 \times 2$ matrix $S$ in the following way
\begin{equation}
	\L = S \Z^{2}.
\end{equation}
The volume of the lattice is given by the absolute value of the determinant of the generating matrix (which is unique) and the density or redundancy is given by the reciprocal of the volume. That is,
\begin{equation}
	\text{vol}(\L) = |(\det(S))| \qquad \text{ and } \qquad \delta(\L) = \frac{1}{\text{vol}(\L)}.
\end{equation}
It is of particular interest to know when a Gabor system is a Gabor frame because in this case there are stable ways to reconstruct a function $f \in \Lt[]$ from the inner products of the function and the atoms of the Gabor system. In order to be a Gabor frame a Gabor system has to fulfil the frame inequality
\begin{equation}\label{eq_frame_inequality}
	A \norm{f}^2 \leq \sum_{\l \in \L} | \langle f, \, \pi(\l) g \rangle |^2 \leq B \norm{f}^2, \qquad \forall f \in \Lt[],
\end{equation}
for some positive constants $0 < A \leq B < \infty$ called frame bounds. Throughout this work, whenever we speak of frame bounds we mean the tightest possible, hence, optimal frame bounds. In the case of an orthonormal basis we have $A=B=1$. To the Gabor system $\G(g,\l)$ we associate the Gabor frame operator
\begin{equation}
	S_g^\L f = \sum_{\l \in \L} \langle f, \, \pi(\l)g \rangle \pi(\l) g, \qquad \forall f \in \Lt[].
\end{equation}
The existence of the upper frame bound guarantees that the operator is bounded and the sharp upper frame bound equals the operator norm of the Gabor frame operator
\begin{equation}
	B = \norm{S_g^\L}_{Op}.
\end{equation}
The existence of the lower frame bound guarantees that the operator is invertible, hence, we can reconstruct $f$ in the following way
\begin{equation}
	f = \sum_{\l \in \L} \langle f, \, \pi(\l)g \rangle \pi(\l) \left(S_g^\L\right)^{-1} g.
\end{equation}

Due to the work of Lyubarskii \cite{Lyu92} and Seip \cite{Sei92} we know that for a Gaussian window any lattice of density $\delta > 1$ generates a Gabor frame for $\Lt[]$. Also, in this case we know that we cannot have a frame for $\delta = 1$ because of the Balian-Low theorem \cite{Bal81, Low85}. This implies that we cannot obtain an orthonormal basis with time-frequency shifted Gaussian windows which makes it interesting to study sharp frame bounds for the Gaussian. For more information about Gabor frames we refer to the classical literature on these topics e.g.\ \cite{Chr03, Fei81, FeiGro97, FeiLue15, Gro01, Hei06} as well as to the survey by Gröchenig \cite{Gro14}.

In what follows we will focus on Gabor frames generated by lattices of fixed density $2n$, $n \in \N$ and the standard Gaussian window
\begin{equation}
	g_0(t) = 2^{1/4} e^{-\pi t^2}
\end{equation}
of $\Lt[]$ unit norm. Within this setting we are interested in the tightest possible bounds and will show that the hexagonal lattice minimizes the upper frame bound.

\section{Chirped Gaussians and Sheared Lattices}\label{sec_Gaussians}
After some preliminaries, we will start to compute frame bounds of Gaussian Gabor frames of redundancy $2n$, $n \in \N$. First of all, we only consider lattices where the generating matrix takes the form
\begin{equation}
	S_\gamma = \left(
	\begin{array}{c c}
		\alpha & 0\\
		\alpha \gamma & \beta
	\end{array}
	\right)
	= \left(
	\begin{array}{c c}
		1 & 0\\
		\gamma & 1
	\end{array}
	\right)
	\left(
	\begin{array}{c c}
		\alpha & 0\\
		0 & \beta
	\end{array}
	\right).
\end{equation}
with $\alpha,\beta > 0$ and $\alpha \beta = \frac{1}{2n}$. Instead of looking at lattices of this type and the standard Gaussian, we can also look at rectangular lattices, i.e., $\gamma = 0$, with chirped Gaussians. A chirped (standard) Gaussian is of the form
\begin{equation}
	g_\gamma(t) = 2^{1/4} e^{\pi i \gamma t^2} e^{-\pi t^2}.
\end{equation}
It is well-known from the general theory on the interplay between the symplectic and the metaplectic group, that the two systems
\begin{equation}\label{eq_equivalent_frames}
	\G\left(g_0,S_\gamma \Z^2\right) \qquad \text{ and } \qquad \G\left(g_{-\gamma},S_0 \Z^2\right)
\end{equation}
possess the same sharp frame bounds \cite{Gos11, Gos15, Fau16, Fol89, Gro01}. Second of all, it is enough to look at lattices of the type $\L = S_\gamma \Z^2$ because any lattice $\L \subset \R^2$ can be represented by $\L = Q S_\gamma \Z^2$, where $Q$ is an orthogonal matrix (QR-decomposition). Again, the rotation imposed by $Q$ does not affect the frame bounds as the Gaussian is an eigenfunction with eigenvalue 1 of the corresponding metaplectic operator, which is the fractional Fourier transform described already in \cite{Alm94} or \cite{Dau88}.

We choose the Gabor system $\G\left(g_{-\gamma}, S_0 \Z^2 \right)$ as object of investigation. Due to Janssen \cite{Jan96}, it is known that the lower and upper frame bound are given by the essential infimum and supremum, respectively, of the Fourier series
\begin{equation}\label{eq_periodization_adjoint}
	F_{g_{-\gamma}}(x,\omega) = \frac{1}{\alpha \beta} \sum_{k,l \in \Z} \left\langle g_{-\gamma} , \, M_{\frac{l}{\alpha}} T_{\frac{k}{\beta}} g_{-\gamma} \right\rangle e^{2 \pi i k x} e^{2 \pi i l \omega}.
\end{equation}
From its definition it is clear that we only need to know $F_{g_{-\gamma}}$ on the unit square, i.e., $(x,\omega) \in [0,1] \times [0,1]$. We compute the inner product
\begin{align}
	\left\langle g_{-\gamma} , \, M_{\frac{l}{\alpha}} T_{\frac{k}{\beta}} g_{-\gamma} \right\rangle & = 2^{1/2} \int_{\R} e^{-\pi i \gamma t^2} e^{-\pi t^2} e^{\pi i \gamma (t-k/\beta)^2} e^{-\pi (t-k/\beta)^2} e^{-2 \pi i l t/\alpha} \, dt\\
	& = 2^{1/2} \int_{\R} e^{-\pi i \gamma \left(t^2 - (t-k/\beta)^2\right)} e^{-\pi \left(t^2 + (t-k/\beta)^2\right)} e^{-2 \pi i l t/\alpha} \, dt\\
	& = 2^{1/2} e^{-\pi \frac{k^2}{\beta^2} (1+ i \gamma)} \frac{1}{\sqrt{2}} \int_\R e^{-\pi t^2} e^{- 2 \pi i \left(\frac{l}{\alpha} - \frac{k}{\beta} (\gamma-i)\right) t/\sqrt{2}} \, dt\\
	& = e^{-\pi \frac{k^2}{\beta^2} (1+ i \gamma)} e^{-\frac{\pi}{2} \left(\frac{l}{\alpha} - \frac{k}{\beta} (\gamma-i)\right)^2}\\
	& = e^{-\pi i \frac{k l}{\alpha \beta}} e^{-\frac{\pi}{2} \left(\frac{k^2}{\beta^2} + \left(\frac{l}{\alpha} + \frac{k \gamma}{\beta}\right)^2\right)}.
\end{align}
All we needed in the computations above were a change of variables and the invariance of the standard Gaussian under the Fourier transform (see e.g.\ \cite{Fol89}). Therefore, we have that
\begin{equation}\label{eq_Fourier_series}
	F_{g_{-\gamma}}(x, \omega) = \frac{1}{\alpha \beta} \sum_{k,l \in \Z} e^{-\pi i \frac{k l}{\alpha \beta}} e^{-\frac{\pi}{2} \left(\frac{k^2}{\beta^2} + \left(\frac{l}{\alpha} + \frac{k \gamma}{\beta}\right)^2\right)} e^{2 \pi i k x} e^{2 \pi i l \omega}.
\end{equation}
which can also be identified as the heat kernel of the flat Laplacian on the torus $\R^2/\L$ where $\L = S_\gamma \Z^2$. We note that
\begin{equation}
	F_{g_{-\gamma}}(x,\omega) = F_{g_\gamma}(-x,\omega) = F_{g_{-\gamma}}(-x,-\omega) = F_{g_\gamma}(x,-\omega).
\end{equation}
Changing the sign of $\gamma$ corresponds to a reflection of the lattice with respect to one of the coordinate axes (it does not matter which one). But, since a reflection matrix has determinant -1 it is not an element of the symplectic group and, therefore, we do not have a corresponding metaplectic operator. For this reason we had to start with the system $\G \left(g_{-\gamma}, S_0 \Z^2 \right)$.

We see that for $(\alpha \beta)^{-1} = 2n$ with $n \in \N$ the factor $e^{-\pi i \frac{k l}{\alpha \beta}}$ in equation \eqref{eq_Fourier_series} vanishes, which is why we consider the case of even redundancy. In this case, the function $F_{g_{-\gamma}}$ takes its maximum whenever $(x, \omega) \in \Z \times \Z$. This implies that in the case of even redundancy the optimal upper frame bound for a Gabor frame with standard Gaussian window is given by the formula
\begin{equation}\label{eq_upper_bound_even}
	B = B(\alpha, \beta, \gamma) = 2n \sum_{k,l \in \Z} e^{-\frac{\pi}{2} \left(\frac{k^2}{\beta^2} + \frac{l^2}{\alpha^2}\right)} e^{-\frac{\pi}{2} \left(\frac{k^2 \gamma^2}{\beta^2} - 2 \frac{k l \gamma}{\alpha \beta}\right)}.
\end{equation}
Our goal is to find the global minimum of this function with respect to the parameters $\alpha$, $\beta$ and $\gamma$. Since the product $\alpha \beta = \frac{1}{2n}$ is fixed, this is a minimization problem in 2 variables. It is similar to the problem recently solved by Faulhuber \& Steinerberger \cite{FauSte16}. Unfortunately, the techniques used in that work do not apply as the double sum does not factor into a product of two sums for $\gamma \neq 0$. Still, the double sum converges very nicely, in particular absolutely. We will now show some properties of $B$.
\begin{proposition}\label{prop_separable_maximal}
	For $\alpha$, $\beta$ fixed with $\alpha \beta = \frac{1}{2n}$, $n \in \N$, $B$ is periodic in $\gamma$ with period $\frac{\beta}{\alpha}$ and symmetric with respect to the points $\frac{\beta}{2 \alpha} \Z$. Furthermore, $B$ takes its global maximum only for $\gamma \in \frac{\beta}{\alpha} \Z$, i.e., for rectangular lattices.
\end{proposition}
\begin{proof}
	The periodicity follows from the according property of the lattice and the symmetry follows from the fact that we can choose the sign of $\gamma$ and the periodicity of the lattice. The property that still needs to be verified is that $B$ assumes its global maximum for $\gamma \in \frac{\beta}{\alpha} \Z$.
	
	We split the double sum in the following way.
	\begin{align}
		B(\gamma) & = 2n \sum_{k \in \Z} \left(e^{-\frac{\pi}{2} \frac{k^2}{\beta^2} \left(1+\gamma^2\right)} \sum_{l \in \Z} e^{-\frac{\pi}{2} \left(\frac{l^2}{\alpha^2} - 2 \frac{k l \gamma}{\alpha \beta}\right)}\right)\\
		& = 2n \sum_{k \in \Z} \left(e^{-\frac{\pi}{2} \frac{k^2}{\beta^2} \left(1+\gamma^2\right)} \sum_{l \in \Z} e^{-\pi n \left(\frac{\beta}{\alpha} l^2 - 2 k l \gamma\right)}\right).
	\end{align}
	We will now use Poisson summation, which in this case involves the Fourier transform of a Gaussian (see e.g.\ \cite{Fol89}), to rewrite the inner sum of the expression. For $k$ fixed we have that
	\begin{equation}
		\sum_{l \in \Z} e^{-\pi n \left(\frac{\beta}{\alpha} l^2 - 2 k l \gamma\right)}
		= e^{\pi n \frac{\alpha}{\beta} \gamma ^2 k^2 } \, \sum_{l \in \Z} e^{-\pi  n \frac{\beta}{\alpha} \left(l - \frac{\alpha}{\beta} \gamma k \right)^2}
		= \sqrt{\frac{\alpha}{n \beta}} \sum_{m \in \Z} e^{-\frac{\pi}{n} \frac{\alpha}{\beta} \left(m^2 - n^2 \gamma ^2 k^2 \right)} e^{-2 \pi i m \frac{\alpha}{\beta} \gamma k}.
	\end{equation}
	Using the fact that $\left( \alpha \beta\right)^{-1} = 2n$ and due to the convergence properties of the double series we can now rewrite it as
	\begin{align}
		B(\gamma) & = 2n \sqrt{\frac{\alpha}{n \beta}} \sum_{k \in \Z} \left(e^{-\frac{\pi}{2} \frac{k^2}{\beta^2} \left(1+\gamma^2\right)} \sum_{l \in \Z} e^{-\frac{\pi}{n} \frac{\alpha}{\beta} \left(l^2 - n^2 \gamma ^2 k^2 \right)} e^{-2 \pi i l \frac{\alpha}{\beta} \gamma k} \right)\\
		& = 2n \sqrt{2} \, \alpha \sum_{k \in \Z}\left( e^{- 2 \pi \alpha^2 n^2 k^2 (1+\gamma^2)} \sum_{l \in \Z} e^{-2 \pi \alpha^2 (l^2 - n^2 k^2 \gamma^2)} e^{-2 \pi i k l \frac{\alpha}{\beta} \gamma} \right)\\
		& = 2n \sqrt{2} \, \alpha \sum_{k,l \in \Z} e^{-2\pi \alpha^2 \left(l^2 + k^2 n^2\right)} e^{-2 \pi i k l \frac{\alpha}{\beta} \gamma}\\
		& = 2n \sqrt{2} \, \alpha \sum_{k,l \in \Z} e^{-2\pi \alpha^2 \left(l^2 + k^2 n^2\right)} \cos\left(2 \pi k l \frac{\alpha}{\beta} \gamma \right).
	\end{align}
	By using the identity $\cos(2x) = \cos(x)^2-\sin(x)^2 = 1 - 2\sin(x)^2$ we get
	\begin{equation}
		B = 2n \sqrt{2} \alpha \sum_{k,l \in \Z} e^{-2\pi \alpha^2 \left(l^2 + k^2 n^2\right)} \left(1 - 2\sin\left(\pi k l \frac{\alpha}{\beta} \gamma\right)^2\right)
	\end{equation}
	and we see that $B(\gamma)$ is maximal for $\gamma \in \frac{\beta}{\alpha} \Z$.
\end{proof}
Proposition \ref{prop_separable_maximal} shows that for even redundancy the upper frame bound of any Gabor frame with rectangular lattice and standard Gaussian window will become smaller by either shearing the lattice or by chirping the window ($\gamma \notin \frac{\beta}{\alpha} \Z$). This gives analytic evidence that the quadratic lattice cannot be optimal for the upper frame bound among general lattices. We will now state a lemma by Montgomery from which we will be able to conclude when $B$ assumes its minimum. The proof of the upcoming result needs a lot of cumbersome computations and estimates. Therefore, we refer to the original work where this lemma has been proved \cite[Lemma 4]{Mon88}.

\begin{lemma}[Montgomery]\label{lemma_Montgomery}
	Let $c > 0$ be fixed, $r \in \left(0, \frac{1}{2} \right)$ and $s \geq \frac{1}{2}$. We define
	\begin{equation}
		\vartheta(r,s;c) = \sum_{k}\left(e^{-c \pi s k^2} \sum_{l \in \Z} e^{-\frac{c \pi}{s} (l + kr)^2} \right).
	\end{equation}
	Then
	\begin{equation}
		\frac{\partial}{\partial r} \vartheta(r,s;c) < 0.
	\end{equation}
\end{lemma}

The variable $r$ in Montgomery's lemma in principle represents the shearing of the lattice whereas the variable $s$ is related to the lattice parameter $\alpha$ (or $\beta$ as one prefers). The parameter $c$ describes the density of the lattice. Fixing $\alpha$ and $\beta$, $B$ only depends on the shearing parameter $\gamma$. By rewriting equation \eqref{eq_upper_bound_even} in the following way
\begin{align}
	B(\gamma) & = 2n \sum_{k,l \in \Z} e^{-\frac{\pi}{2}\left(\frac{l}{\alpha} - \frac{k \gamma}{\beta}\right)^2} e^{-\frac{\pi}{2} \frac{k^2}{\beta^2}}\\
	& = 2n \sum_{k \in \Z} \left( e^{-2 \pi \alpha^2 n^2 k^2} \sum_{l \in \Z} e^{-\frac{\pi}{2 \alpha^2}\left(l - \gamma \frac{\alpha}{\beta} k \right)^2}\right)\\
	& = 2n \, \vartheta \left(-\gamma \frac{\alpha}{\beta}, 2 \alpha^2 n; n \right),
\end{align}
Montgomery's lemma implies the following proposition.
\begin{proposition}\label{prop_optimal_shear}
	For $\alpha$, $\beta$ fixed with $\alpha \geq \frac{1}{\sqrt{2}\sqrt{2n}}$ and $\alpha \beta = \frac{1}{2n}$, $n \in \N$, $B$ assumes its global minimum only for $\gamma \in \frac{\beta}{\alpha} \left(\frac{1}{2} + \Z\right)$.
\end{proposition}

\section{Minimal Theta Functions}\label{sec_minimal}

In this section we will study connections between theta functions, quadratic forms and lattices. For more information we refer to the textbook by Conway \& Sloane \cite{ConSlo99}. We start with some definitions in order to state a theorem formulated by Montgomery \cite{Mon88}. For $\rho > 0$ and a positive definite quadratic form $q(u_1,u_2) = a u_1^2 + b u_1 u_2 + c u_2^2$ with discriminant
\begin{equation}
	D = b^2 - 4 ac = -1
\end{equation}
we define the theta function
\begin{equation}\label{eq_theta_q}
	\theta_q(\rho) = \sum_{k,l \in \Z} e^{-2\pi \rho \, q(k,l)}.
\end{equation}
The theta function satisfies the identity
\begin{equation}\label{eq_identity}
	\theta_q(\rho) = \frac{1}{\rho} \, \theta_q\left(\frac{1}{\rho}\right),
\end{equation}
which can be looked up in \cite[equation (1)]{Mon88} or can be verified by using the 2-dimensional Poisson summation formula. Also, we can associate the symmetric matrix
\begin{equation}
	G_q = \left(
	\begin{array}{c c}
		a & b/2\\
		b/2 & c
	\end{array}
	\right)
\end{equation}
to the quadratic form $q(u_1,u_2) = (u_1,u_2) \cdot G_q \cdot (u_1,u_2)^T$. The index $q$ indicates that the matrix is associated to a quadratic form. Later on, we will also use other indices to emphasize the connection to either a quadratic form or a lattice. We observe that $4 \det(S_q) = -D$.
\begin{theorem}[Montgomery]\label{theorem_Montgomery}
	Let $h(u_1,u_2) = \frac{1}{\sqrt{3}} \left(u_1^2 +u_1 u_2 +u_2^2 \right)$. For any $\rho > 0$ and any positive definite quadratic form $q(u_1,u_2)$ with discriminant $D = -1$ we have
	\begin{equation}
		\theta_q(\rho) \geq \theta_h(\rho).
	\end{equation}
	If we have equality for some $\rho > 0$, then $q$ and $h$ are equivalent forms and $\theta_q \equiv \theta_h$.
\end{theorem}
The quadratic form $h(u_1,u_2)$ in Montgomery's theorem is associated to the hexagonal lattice. This shows, that among all theta functions associated to a quadratic form with fixed discriminant, the form associated to a hexagonal lattice minimizes the theta function.

As a next step, we will show that the quadratic form associated to the chosen lattice appears directly in the exponent of the theta function describing the upper frame bound. In fact, we will see that for the standard Gaussian window and even redundancy, the upper frame bound is given by sampling and adding the values of the ambiguity function with respect to the lattice. Before doing so, we formulate the definition of the adjoint lattice, a proposition about the frame operator and introduce the ambiguity function.
\begin{definition}[Adjoint lattice]
	For a lattice $\Lambda \subset \R^2$ the adjoint lattice is given by
	\begin{equation}
		\Lambda^\circ = \text{vol}(\Lambda)^{-1} \Lambda.
	\end{equation}
\end{definition}
With this definition we can formulate the following proposition which holds for any $g \in \Lt[]$ and is also known as Janssen's representation.
\begin{proposition}[Janssen's representation]\label{proposition_Janssen_representation}
	Let $g \in \Lt[]$ with the property that
	\begin{equation}\label{eq_Janssen_coefficients}
		\sum_{\lambda^\circ \in \Lambda^\circ} |\langle g, \pi(\lambda^\circ) g \rangle | < \infty.
	\end{equation}
	Then, the frame operator can be written in the form
	\begin{equation}\label{eq_Janssen_operator}
		S^\Lambda_{g} = \emph{vol}(\Lambda)^{-1} \sum_{\lambda^\circ \in \Lambda^\circ} \langle g, \pi(\lambda^\circ) g \rangle \pi(\lambda^\circ).
	\end{equation}
\end{proposition}
For a Gabor frame with a lattice $\Lambda$ of density $2n$, $n \in \N$ and standard Gaussian window $g_0$ we see that the $\ell^1(\L^\circ)$-norm of the coefficients in Janssen's representation of the frame operator \eqref{eq_Janssen_operator}, which coincides (up to the factor vol$(\L)^{-1}$) with Equation \eqref{eq_Janssen_coefficients}, gives us the upper frame bound by setting $(x,\omega) = (0,0)$ in Equation \eqref{eq_periodization_adjoint}.

We will now show that for Gabor frames with standard Gaussian window $g_0$ and general lattices the upper frame bound is given (up to a factor 2) by the $\ell^1(2 \Lambda)$-norm of the samples of the ambiguity function $\A g_0$ of the standard Gaussian.
\begin{definition}[Ambiguity function]\label{AmbiguityFunction}
	The ambiguity function of a function $f \in \Lt[]$ is given by
	\begin{equation}
		\A f(x,\omega) = \int_{\R^d} f\left(t+\frac{x}{2}\right) \overline{f\left(t-\frac{x}{2}\right)} e^{-2 \pi i \omega \cdot t} \, dt.
	\end{equation}
\end{definition}

The ambiguity function measures how concentrated a function is in the time-frequency plane. For the standard Gaussian the ambiguity function is given by
\begin{equation}
	\A g_0(x, \omega) = e^{-\frac{\pi}{2} \left(x^2 + \omega^2\right)}.
\end{equation}

For $\L = S_\gamma \Z^2 = \left(
\begin{array}{c c}
	\alpha & 0\\
	\alpha \gamma & \beta
\end{array}
\right) \Z^2$, the Gram matrix associated to $\L$ is given by
\begin{equation}\label{eq_Gram}
	G_\L = \left(
	\begin{array}{c c}
		\left(1 + \gamma^2 \right)\alpha^2 & \alpha \beta \gamma \\
		\alpha \beta \gamma & \beta^2
	\end{array}
	\right)
\end{equation}
with determinant $\det(G_\L) = \frac{1}{\alpha^2 \beta ^2} = \frac{1}{4n^2} = -\frac{1}{4}D$, where $D$ is the discriminant of the associated quadratic form. The adjoint lattice, its associated matrix and Gram matrix are given by
\begin{equation}
	\L^\circ = S_\gamma^\circ \Z^2 = \left(
	\begin{array}{c c}
		\frac{1}{\beta} & 0\\
		\frac{\gamma}{\beta} & \frac{1}{\alpha}
	\end{array}
	\right)\Z^2 \quad \text{ and } \quad
	G_{\L^\circ} = \left(
	\begin{array}{c c}
		\frac{1 + \gamma^2}{\beta^2} & \frac{\gamma}{\alpha \beta} \\
		\frac{\gamma}{\alpha \beta} & \frac{1}{\alpha^2}
	\end{array}
	\right).
\end{equation}
For $\G(g_0,\L)$, with vol$(\L)^{-1} = 2n$, $n \in \N$, we compute the upper frame bound as
\begin{align}
	B = B_{g_0,2n}(\Lambda) & = 2n \sum_{k,l \in \Z} e^{-\frac{\pi}{2} \left(\frac{1+\gamma^2}{\beta^2}k^2 + 2 \frac{\gamma k l}{\alpha \beta} + \frac{1}{\alpha^2}l^2 \right)}\\
	& = \det (G_{\L^\circ})^{1/2} \sum_{k,l \in \Z} e^{-\frac{\pi}{2}\left\langle S^\circ (k,l),\, S^\circ (k,l) \right\rangle}\\
	& = 2(-D)^{-1/2} \sum_{k,l \in \Z} e^{-2 \pi (-D)^{-1/2} q_\L(k,l)}\\
	& \overset{\eqref{eq_identity}}{=} 2 \sum_{k,l \in \Z} e^{-\frac{\pi}{2} q_\L(2k,2l)}\\
	& = 2 \sum_{\lambda \in \Lambda} Ag_0(2\lambda).
\end{align}
Here, $q_\L(k,l) = (-D)^{-1/2} \left\langle S (k,l),\, S (k,l) \right\rangle$ is a quadratic form with discriminant $-1$. Summing up the results we find out that
\begin{equation}
	B_{g_0,2n}(\Lambda) = 2 \, \theta_{q_\Lambda} \left( \frac{1}{n} \right) = 2n \, \theta_{q_\Lambda} (n).
\end{equation}
Therefore, Montgomery's theorem (Theorem \ref{theorem_Montgomery}) implies that for a Gabor frame with Gaussian window and any lattice of even redundancy, a lattice with associated Gram matrix
\begin{equation}
	G_h = S_h^T S_h = \frac{1}{2n} \frac{1}{\sqrt{3}}\left(
	\begin{array}{c c}
		2 & 1\\
		1 & 2
	\end{array}
	\right)
\end{equation}
yields the minimal Gabor frame operator norm or equivalently the smallest possible upper frame bound. From the entries of the Gram matrix $G_\L$ in equation \eqref{eq_Gram} and the condition $(\alpha \beta)^{-1} = 2n$, $n \in \N$ we determine the lattice parameters $\alpha$, $\beta$ and $\gamma$ (up to a sign which does not affect the geometry of the lattice or the frame bounds). We get
\begin{equation}
	\alpha = \frac{1}{\sqrt{2n}} \frac{\sqrt[4]{3}}{\sqrt{2}}, \quad \beta = \frac{1}{\sqrt{2n}} \frac{\sqrt{2}}{\sqrt[4]{3}}, \quad \gamma = \frac{1}{\sqrt{3}}
\end{equation}
and hence
\begin{equation}
	S_h = Q \frac{1}{\sqrt{2n}}
	\left(
	\begin{array}{c c}
		\frac{\sqrt[4]{3}}{\sqrt{2}} & 0\\
		\frac{\sqrt{2}}{2 \sqrt[4]{3}} & \frac{\sqrt{2}}{\sqrt[4]{3}}
	\end{array}
	\right).
\end{equation}
The matrix $Q$ is an orthogonal matrix which does not affect the Gram matrix $G_h$ since $Q^T Q = I$. Therefore, only rotated versions of the hexagonal lattice give the optimal upper frame bound. Any other quadratic form which also minimizes the theta function $\theta_q(\rho)$ is as well associated to a version of the hexagonal lattice, i.e., a version of the hexagonal lattice with another choice of basis and Gram matrix
\begin{equation}
	\widetilde{G}_h = \mathcal{B}^T G_h \, \mathcal{B}, \quad \mathcal{B} \in SL(2,\Z),
\end{equation}
which proves our main result Theorem \ref{theorem_main}.

\section{A Result for Redundancy 2}\label{sec_red2}

We want to make a final remark on frame bounds of the Gabor frame $\G(g_0,\L)$ for redundancy 2. If we take the standard Gaussian and the square lattice $\L_\square = \frac{1}{\sqrt{2}} \Z \times \frac{1}{\sqrt{2}} \Z$ we find out that the sharp lower and upper frame bound are given by
\begin{equation}\label{eq_lower_bound_2}
	A_{g_0,2}(\L_\square) = 2 \left(\sum_{k \in \Z} (-1)^k e^{-\pi k^2} \right)^2 = 2 \left(\frac{\pi^{1/4}}{2^{1/4} \Gamma\left(\frac{3}{4} \right)}\right)^2
\end{equation}
\begin{equation}\label{eq_upper_bound_2}
	B_{g_0,2}(\L_\square) = 2 \left(\sum_{k \in \Z} e^{-\pi k^2} \right)^2 = 2 \left(\frac{\pi^{1/4}}{\Gamma\left(\frac{3}{4} \right)}\right)^2,
\end{equation}
where $\Gamma$ is the usual {gamma function}\index{gamma function} defined as
\begin{equation}
	\Gamma(t) = \int_{\R_+} x^{t-1} e^{-x} \, dx.
\end{equation}
The two formulas for $A_{g_0,2}(\L_\square)$ and $B_{g_0,2}(\L_\square)$ follow from classical results about Jacobi's theta-3 and theta-4 function. For $s > 0$, Jacobi's theta functions are given by
\begin{align}
	\theta_2(s) & = \sum_{k \in \Z} e^{-\pi \left(k-\frac{1}{2}\right)^2 s},\\
	\theta_3(s) & = \sum_{k \in \Z} e^{-\pi k^2 s},\\
	\theta_4(s) & = \sum_{k \in \Z} (-1)^k e^{-\pi k^2 s}.
\end{align}
A version of equation \eqref{eq_lower_bound_2} is also given in \cite[equation (28)]{Mez13}. By using the identities
\begin{equation}
	\theta_3(s)^4 = \theta_2(s)^4 + \theta_4(s)^4,
\end{equation}
and
\begin{equation}
	\theta_2(s) = \frac{1}{\sqrt{s}} \, \theta_4(s)
\end{equation}
(see e.g.\ \cite[Chap.\ 4.4]{ConSlo99}) we find out that
\begin{equation}
	B_{g_0,2}(\L_\square)^2 = 4 \theta_3(1)^4 = 8 \theta_4(1)^4 = 2 A_{g_0,2}(\L_\square)^2.
\end{equation}
Equation \eqref{eq_upper_bound_2} follows now from equation \eqref{eq_lower_bound_2}. It is easy to see that the condition number is $B_{g_0,2}(\L_\square)/A_{g_0,2}(\L_\square) = \sqrt{2}$ which was also observed by Strohmer \& Beaver in \cite{StrBea03}. Furthermore, they mentioned that in the case of the hexagonal lattice of redundancy 2, the condition number is approximately $B_{g_0,2}(\L_h)/A_{g_0,2}(\L_h) \approx 1.2599$ which, as they say, `is suspiciously close to $\sqrt[3]{2}$'. We will now prove that it is exactly the conjectured value.

Taking the quadratic form $h(u_1,u_2) = \frac{1}{\sqrt{3}} \left( u_1^2 - u_1 u_2 +u_2^2 \right)$, which is equivalent to the quadratic form in Theorem \ref{theorem_Montgomery}, an adaption of a result by Baernstein \cite{Bae97} shows that the series
\begin{equation}
	F_h(x,\omega;\rho) = \sum_{k,l \in \Z} e^{-2 \pi \rho \, h(k,l)} e^{2 \pi i (kx + l \omega)}
\end{equation}
assumes its minimum at $(x,\omega) = \left( \frac{1}{3}, \frac{1}{3} \right)$ and, due to symmetry, also at $(x,\omega) = \left( \frac{2}{3}, \frac{2}{3} \right)$ and of course at all integer shifts of these points. Therefore, the lower and upper frame bound of the Gabor system $\G(g_0, \L_h)$, where $\L_h$ is a hexagonal lattice of redundancy 2, are given by
\begin{equation}
	A_{g_0,2}(\L_h) = 2 \sum_{k,l \in \Z} e^{-2 \pi \frac{1}{\sqrt{3}} \left( k^2 - k l + l^2 \right)} e^{2 \pi i \left(k + l\right)/3} = 2 \sum_{k,l \in \Z} e^{-2 \pi \frac{1}{\sqrt{3}} \left( k^2 + k l + l^2 \right)} e^{2 \pi i \left(k - l\right)/3}
\end{equation}
and
\begin{equation}
	B_{g_0,2}(\L_h) = 2 \sum_{k,l \in \Z} e^{-2 \pi \frac{1}{\sqrt{3}} \left( k^2 + k l + l^2 \right)}.
\end{equation}
\begin{proposition}\label{prop_optimal_condition}
	For redundancy 2, the ratio of the frame bounds for a Gabor frame with standard Gaussian window and a hexagonal lattice is given by
	\begin{equation}
		\frac{B_{g_0,2}(\L_h)}{A_{g_0,2}(\L_h)} = \sqrt[3]{2}.
	\end{equation}
\end{proposition}
\begin{proof}
	We will use a result on cubic theta functions by Hirschhorn, Garvan \& Borwein derived in 1993 \cite{HirGarBor93} to prove the statement. To stick close to their notation we introduce the following functions
	\begin{align}
		a(q) & = \sum_{k,l \in \Z} q^{k^2 + kl + l^2}\\
		b(q) & = \sum_{k,l \in \Z} q^{k^2 + kl + l^2} \zeta_3^{k-l}\\
		c(q) & = \sum_{k,l \in \Z} q^{\left(k+\frac{1}{3}\right)^2 + \left(k+\frac{1}{3}\right)\left(l+\frac{1}{3}\right) + \left(l+\frac{1}{3}\right)^2}
	\end{align}
	where $\zeta_3^3 = 1$ and $\zeta_3 \neq 1$ and $|q| < 1$. These functions fulfil the identity
	\begin{equation}
		a(q)^3 = b(q)^3 + c(q)^3
	\end{equation}
	(see \cite[equation (1.8)]{HirGarBor93}). Setting $q = e^{-2 \pi/\sqrt{3}}$ we will now prove that actually $b\left(e^{-2 \pi/\sqrt{3}}\right) = c\left(e^{-2 \pi/\sqrt{3}}\right)$ by using Poisson summation. We start with the observation that $e^{-\frac{2 \pi}{\sqrt{3}} \left( k^2 + kl +l^2\right)}$ and $e^{-\frac{2 \pi}{\sqrt{3}} \left( k^2 - kl +l^2\right)}$ are the 2-dimensional Fourier transforms of each other which is best confirmed by using Folland's result on the Fourier transform of Gaussians \cite[App.\ A, Theorem 1]{Fol89}. Therefore, by using the 2-dimensional Poisson summation formula we have 
	\begin{equation}
		\sum_{k,l \in \Z} e^{-\frac{2 \pi}{\sqrt{3}} \left( k^2 - kl + l^2\right)} e^{-2 \pi i \frac{(k+l)}{3}}
		= \sum_{k,l \in \Z} e^{-\frac{2 \pi}{\sqrt{3}} \left( \left(k+\frac{1}{3}\right)^2 + \left(k+\frac{1}{3}\right)\left(l+\frac{1}{3}\right) + \left(l+\frac{1}{3}\right)^2 \right)}.
	\end{equation}
	Hence, it follows that
	\begin{align}
		A_{g_0,2}(\L_h) & = 2 \, b\left(e^{-2 \pi/\sqrt{3}}\right) = 2 \, c\left(e^{-2 \pi/\sqrt{3}}\right)\\
		B_{g_0,2}(\L_h) & = 2 \, a\left(e^{-2 \pi/\sqrt{3}}\right)
	\end{align}
	which gives
	\begin{equation}
		B_{g_0,2}(\L_h)^3 = 2 \, A_{g_0,2}(\L_h)^3
	\end{equation}
	and the proof is finished.
\end{proof}
The results in this section give the first analytic proof that, for a Gabor frame with standard Gaussian window and a lattice of redundancy 2, the hexagonal lattice yields a better frame condition number than the square lattice.

\bibliographystyle{plain}

\end{document}